\documentclass[12pt,reqno]{amsart}

\newcommand\version{November 17, 2024}

\usepackage{amsmath,amsfonts,amsthm,amssymb,amsxtra}
\usepackage{color}
\usepackage{bbm} 
\usepackage{hyperref} 

\newtheorem{theorem}{Theorem}

\newtheorem{lemma}[theorem]{Lemma}

\theoremstyle{definition}

\theoremstyle{remark}

\newtheorem{remark}[theorem]{Remark}

\newcommand{\C}{\mathbb{C}}

\renewcommand{\epsilon}{\varepsilon}

\newcommand{\N}{\mathbb{N}}

\renewcommand{\phi}{\varphi}
\newcommand{\R}{\mathbb{R}}

\newcommand{\Sph}{\mathbb{S}}

\DeclareMathOperator{\re}{Re}

\begin{document}

\title[Dirichlet and Neumann eigenvalues on Carnot groups --- \version]{Inequalities between Dirichlet and Neumann eigenvalues on Carnot groups}

\author{Rupert L. Frank}
\address[Rupert L. Frank]{Mathe\-matisches Institut, Ludwig-Maximilians Universit\"at M\"unchen, The\-resienstr.~39, 80333 M\"unchen, Germany, and Munich Center for Quantum Science and Technology, Schel\-ling\-str.~4, 80799 M\"unchen, Germany, and Mathematics 253-37, Caltech, Pasa\-de\-na, CA 91125, USA}
\email{r.frank@lmu.de}

\author{Bernard Helffer}
\address[Bernard Helffer]{Laboratoire de Math\'ematiques Jean Leray, CNRS, Nantes Universit\'e, F44000. Nantes, France}
\email{Bernard.Helffer@univ-nantes.fr}

\author[A. Laptev]{Ari Laptev}
\address[Ari Laptev]{Imperial College London, 180 Queen's Gate, London SW7 2AZ, UK, and Sirius Mathematics Center, Sirius University of Science and Technology, 1 Olympic Ave, 354340, Sochi, Russia}
\email{a.laptev@imperial.ac.uk}

\thanks{\copyright\, 2024 by the authors. This paper may be reproduced, in its entirety, for non-commercial purposes.\\
	Partial support through US National Science Foundation grant DMS-1954995 (R.L.F.), as well as through the German Research Foundation through EXC-2111-390814868 and TRR 352-Project-ID 470903074 (R.L.F.) is acknowledged.}

\begin{abstract}
	We show that the $j$-th Dirichlet eigenvalue of the sub-Laplacian on an open set of a Carnot group is greater than the $(j+1)$-st Neumann eigenvalue. This extends earlier results in the Euclidean and Heisenberg case and has a remarkably simple proof.
\end{abstract}

\dedicatory{To Nikolai Nadirashvili, on the occasion of his 70th birthday}

\maketitle

\section{Introduction and main result}

Let $G$ be a Carnot (also known as stratified) group. That is, $G$ is a connected and simply connected Lie group such that its Lie algebra $\mathfrak g$ admits, for some $r\in\N$, a direct sum decomposition
$$
\mathfrak g = V_1 \oplus V_2 \oplus \ldots \oplus V_r \,,
$$
where $V_j = [V_1,V_{j-1}]$ for $j=2,\ldots,r$ and $[V_1,V_r]=\{0\}$. These conditions are equivalent to the existence of a direct sum decomposition $\mathfrak g = V_1 \oplus V_2 \oplus \ldots \oplus V_r$, where $[V_i,V_j] \subset V_{i+j}$ if $i+j\leq r$ and $[V_i,V_j]=\{0\}$ if $i+j> r$, together with the assumption that $V_1$ generates $\mathcal G$ by its brackets; see, e.g., \cite[Subsection 2.3, Ex.~4]{BoLaUg}. Background on Carnot groups can be found, for instance, in \cite{BoLaUg}.

As usual, we also think of $\mathfrak g$ as the algebra of left-invariant vector fields on $G$. We fix a left-invariant inner product on $V_1$ and consider the corresponding sub-Laplacian
$$
- \Delta := - \sum_{\ell=1}^m X_\ell^2 \,,
$$
where $(X_1,\ldots,X_m)$ is an orthonormal basis of $V_1$. The definition of $-\Delta$ does not depend on the choice of the orthonormal basis. Given an open set $\Omega\subset G$ we are interested in the Dirichlet and Neumann realizations, denoted by $-\Delta_\Omega^{\rm D}$ and $-\Delta_\Omega^{\rm N}$, of the sub-Laplacian in $\Omega$. These are selfadjoint, nonnegative, unbounded operators in $L^2(\Omega)$, where the underlying measure is Haar measure. The Dirichlet and Neumann sub-Laplacians are defined using the method of quadratic forms; see Section~\ref{sec:proof} for details. Under the assumption
\begin{equation}
	\label{eq:ass}
	\text{the operator}\ -\Delta_\Omega^{\rm N} \ \text{has compact resolvent}
\end{equation}
the spectra of both operators $-\Delta_\Omega^{\rm D}$ and $-\Delta_\Omega^{\rm N}$ consist of eigenvalues of finite multiplicities, accumulating at infinity only. We enumerate them in nondecreasing order and taking multiplicities into account by $\lambda_j(-\Delta_\Omega^{\rm D})$ and $\lambda_j(-\Delta_\Omega^{\rm N})$ with $j\in\N=\{1,2,3,\ldots\}$. It follows immediately from the variational principle that
$$
\lambda_j(-\Delta_\Omega^{\rm D}) \geq \lambda_j(-\Delta_\Omega^{\rm N})
\qquad\text{for all}\ j\in\N \,.
$$ 
Clearly, when $G=\R$, then $\lambda_j(-\Delta_\Omega^{\rm D}) \geq \lambda_{j+1}(-\Delta_\Omega^{\rm N})$ with equality when $\Omega$ is an interval. In this paper we show that, as soon as $G\neq\R$, then a similar inequality holds and is strict for any domain. Here is the precise statement.

\begin{theorem}
	\label{main}
	Let $G$ be a Carnot group with $\dim V_1\geq 2$ and let $\Omega\subset G$ be a nonempty open set satisfying \eqref{eq:ass}. Then
	$$
	\lambda_j(-\Delta_\Omega^{\rm D}) > \lambda_{j+1}(-\Delta_\Omega^{\rm N})
	\qquad\text{for all}\ j\in\N \,.
	$$
\end{theorem}

In case $G=\R^n$ with its usual addition this theorem is due to Friedlander \cite{Fri} under some additional regularity assumptions on $\Omega$. An elegant alternative proof, which also removes these additional regularity assumptions, was found by Filonov \cite{Fi}.

In case $G$ is the (one- or higher-dimensional) Heisenberg group this theorem is due to two of us \cite{FrLa}. Our earlier proof, which adapted Filonov's strategy, was based on a nontrivial identity for the spectral projection kernel of the Landau Hamiltonian; see also \cite{Fr}.

Our contribution in the present paper is a considerable simplification of the proof for the Heisenberg group and its extension to any Carnot group. No particular identity is needed anymore.

In fact, our method of proof works for a more general class of operators than sub-Laplacians on Carnot groups. We will present this general setting in the final section of this paper.

We emphasize that there are no assumptions on $\Omega$ except for \eqref{eq:ass}. The latter assumption is equivalent to the compactness of the embedding $S^1(\Omega)\subset L^2(\Omega)$. Sufficient conditions for this can be found, for instance, in \cite{GaNh}. For a discussion of issues related to the boundary regularity in the sub-Riemannian case we refer to \cite{Kl,FrHe} and the references therein. Moreover, we recall that Filonov's proof \cite{Fi} on the standard additive group $\R^n$ relies on the unique continuation theorem. In the context of the Heisenberg group we have shown in \cite{FrLa} how to avoid the use of such a theorem. The same applies in the present, more general setting. We note, however, that in our case the vector fields have analytic coefficients and consequently the unique continuation theorem holds \cite{Bo}. Finally, the results in \cite{Fi,FrLa} are proved under the assumption that the open set has finite measure. As we show here, this is a consequence of~\eqref{eq:ass}.

Let us provide some historical context concerning bounds similar to those in Theorem \ref{main}. When $G=\R^n$ with its usual addition such bounds were first proved by Aviles \cite{Av} when $\Omega$ is mean convex, before Payne \cite{Pa} for $n=2$ and Levine and Weinberger \cite{LeWe} for general $n\geq 2$ proved the bound $\lambda_j(-\Delta_\Omega^{\rm D})\geq \lambda_{j+n}(-\Delta_\Omega^{\rm N})$ when $\Omega$ is convex. Whether these inequalities with an index shift by $n$ remain valid for nonconvex $\Omega$ is an open question. (The calculations in the ``counterexample'' in \cite[p.~207]{LeWe} seem to be erroneous, as was pointed out to us by M.~Levitin and I.~Polterovich.) An affirmative answer was given recently for $n=2$ when $\Omega$ is simply connected \cite{Ro}. 

Mazzeo \cite{Ma} proposed to study similar inequalities between Dirichlet and Neumann eigenvalues on Riemannian manifolds and he showed that inequalities as in Theorem~\ref{main} are valid for domains $\Omega$ in a Riemannian symmetric space of noncompact type, for instance in hyperbolic space. They are not valid for all domains on all closed manifolds, however, as shown in \cite{ArMa}. These results have spawned a large literature that we cannot review here and we refer to the above mentioned paper for further references. The paper \cite{Ha} was the first to obtain a partial result in the subelliptic context. It is an \emph{open problem} to prove an inequality between Dirichlet and Neumann eigenvalues with an index shift of two or larger in the non-Euclidean setting, for instance under geometric or topological assumptions on the underlying domain.


\section{Proof}\label{sec:proof}

As before, let $G$ be a Carnot group with a fixed left-invariant inner product on the first stratum $V_1$. For an open set $\Omega\subset G$ the Folland--Stein--Sobolev space $S^1(\Omega)$ consists of all $u\in L^2(\Omega)$ for which the distributions $Xu$ belong to $L^2(\Omega)$ for all $X\in V_1$, normed by $\sqrt{ \| |\nabla u| \|_{L^2}^2 + \|u\|^2_{L^2}}$. Here we have set
$$
|\nabla u| := \sqrt{ \sum_{\ell=1}^m |X_\ell u|^2 } \,,
$$
which is independent of the choice of the orthonormal basis $(X_1,\ldots,X_m)$ of $V_1$. We define $S^1_0(\Omega)$ as the $S^1(\Omega)$-closure of functions in $S^1(\Omega)$ that vanish outside of a compact subset of $\Omega$. We denote by $S^1_{\rm loc}(G)$ the set of functions $u$ on $G$ such that $u|_\Omega\in S^1(\Omega)$ for all open, relatively compact $\Omega\subset G$.

The nonnegative quadratic form $u\mapsto \| |\nabla u| \|_{L^2}^2$ is closed in $L^2(\Omega)$ when considered with either one of the form domains $S^1(\Omega)$ and $S^1_0(\Omega)$. Consequently (see, e.g., \cite[Theorem 1.18]{FrLaWe}), it gives rise to selfadjoint, nonnegative operators $-\Delta_\Omega^{\rm N}$ and $-\Delta_\Omega^{\rm D}$ with form domains $S^1(\Omega)$ and $S^1_0(\Omega)$, respectively. As is well known (see, e.g., \cite[Corollary 1.21]{FrLaWe}), assumption \eqref{eq:ass} is equivalent to the compactness of the embedding $S^1(\Omega)\subset L^2(\Omega)$. We shall need the following fact.

\begin{lemma}\label{compactness}
	Let $\Omega\subset G$ be open and assume that $S^1(\Omega)\subset L^2(\Omega)$ is compact. Then $\Omega$ has finite (Haar) measure.
\end{lemma}

\begin{proof}
	The proof is essentially the same as that of \cite[Lemma 2.42]{FrLaWe}, which concerns the case where $G=\R^n$ with its usual addition. To adapt this proof to the setting of general Carnot groups we replace the Euclidean distance $|x|$ from the origin by the Carnot--Carath\'eodory distance $d(x,0)$ from the unit element $0$ of $G$. For background on this distance we refer to \cite[Subsection~5.2]{BoLaUg}. The first step of the proof of \cite[Lemma 2.42]{FrLaWe} depends only on the homogeneity of the distance and the fact that the volume of a ball grows polynomially with the distance. Both properties are valid in present setting. The second step of the proof depends on the construction of a family of trial functions. The fact that these functions belong to $S^1(\Omega)$ follows from \cite[Theorem~1.3]{GaNh2}. In view of the eikonal equation $|\nabla d(\cdot,0)|=1$ almost everywhere \cite[Theorem 3.1]{MoSC}, we can argue in the same way as in \cite[Lemma 2.42]{FrLaWe}.
\end{proof}

We now come to the proof of Theorem \ref{main}, for which we adapt Filonov's method \cite{Fi}. The crucial ingredient is the following result.

\begin{lemma}\label{extrafunction}
	Let $G$ be a Carnot group with $\dim V_1\geq 2$ and let $\lambda>0$. Then there are infinitely many linearly independent functions $U\in S^1_{\rm loc}(G)\cap L^\infty(G)$ such that $-\Delta U = \lambda U$ and $|\nabla U|^2 = \lambda |U|^2$ in $G$.
\end{lemma}

The function $U$ will be smooth, but for the proof of Theorem \ref{main} it suffices that the equation $-\Delta U = \lambda U$ holds in the sense of distributions.

\begin{proof}
	For the proof of this lemma we choose a particular form of $G$. Indeed, having fixed the basis $(X_1,\ldots,X_m)$ of $V_1$, we know that $G$ is (Lie-group isomorphic to)
	$$
	\R^{m_1}\times\cdots\times \R^{m_r}
	$$
	with $m_\rho=\dim V_\rho$ for $\rho=1,\ldots,r$. Writing $x=(x^{(1)},\ldots,x^{(r)})$, $y=(y^{(1)},\ldots,y^{(r)})$ with $x^{(\rho)},y^{(\rho)}\in\R^{m_\rho}$ for $\rho=1,\ldots,r$, the group law is given by
	$$
	(x\circ y)^{(\rho)} = x^{(\rho)} + y^{(\rho)} + Q^{(\rho)}(x^{(1)},\ldots,x^{(\rho-1)},y^{(1)},\ldots,y^{(\rho-1)}) \,,
	$$
	where $Q^{(1)}=0$ and where, for $\rho=2,\ldots,r$, $Q^{(\rho)}$ is a function taking values in $\R^{m_\rho}$. (In fact, each component of $Q^{(\rho)}$ is a polynomial of a certain degree, but this is not relevant for us.) The vector fields $X_1,\ldots,X_m$ (note $m=m_1$) are given by
	\begin{equation}\label{eq:new}
	X_\ell = \frac{\partial}{\partial x^{(1)}_\ell} + \sum_{\rho=2}^r \sum_{k=1}^{m_\rho} \frac{\partial Q^{(\rho)}_\alpha}{\partial y_\ell^{(1)}}\Big|_{y=0} \frac{\partial}{\partial x^{(\rho)}_k}
	\end{equation}
	and the Lebesgue measure on $\R^N$ is bi-invariant and, consequently, a Haar measure. These facts are well known and are proved, for instance, in \cite[Subsection 2.2]{BoLaUg}; see also \cite{FoSt,HeNo}.
	
	What will be important in what follows is that $X_\ell$ acts as $\frac{\partial}{\partial x^{(1)}_\ell}$ on functions that only depend on the variable $x^{(1)}$. Consequently, for any fixed $\omega\in\Sph^{m-1}$ the function	
	$$
	U(x) = e^{i\sqrt{\lambda}\, \omega\cdot x^{(1)}}
	$$
	satisfies the claimed properties. For different $\omega$'s these functions are linearly independent and, since $m=\dim V_1\geq 2$, there are infinitely many such functions. This proves the assertion of the lemma.
\end{proof}

\begin{remark}
	In the case of the Heisenberg group $\mathbb H_n$ with variables $(x,y,t)\in\R^{n}\times\R^n\times\R$ and vector fields 
	$$
	X_\ell=\frac{\partial}{\partial x_\ell} + \frac12 y_\ell \frac{\partial}{\partial t} \,,
	\qquad
	Y_\ell=\frac{\partial}{\partial y_\ell} - \frac12 x_\ell \frac{\partial}{\partial t} \,,
	\qquad
	\text{for}\ \ell=1,\ldots,2n\,,
	$$
	spanning $V_1$ the group is already in the appropriate form and we can take $$. U(x,y,t) = e^{i\sqrt\lambda\, \omega\cdot(x,y)}$$with $\omega\in\Sph^{2n-1}$. This is much simpler than the choice in \cite{FrLa}.
\end{remark}

With Lemma \ref{extrafunction} at hand we can prove our main result in the same way as in \cite{FrLa}. We include the details for the sake of completeness.

\begin{proof}[Proof of Theorem \ref{main}]
	Let $\Omega\subset G$ be a nonempty open set satisfying \eqref{eq:ass}. We fix $j\in\N$ and denote by $\mathcal N$ the space spanned by all Neumann eigenfunctions corresponding to eigenvalues $\leq\lambda_{j+1}^{\rm N}$. Here and in what follows we denote $\lambda_k^\#:=\lambda_k(-\Delta_\Omega^\#)$ for $\#\in\{\rm D, N\}$. (We emphasize the dimension of $\mathcal N$ might exceed $j+1$ if $\lambda_{j+1}^{\rm N}$ is degenerate.) We choose orthonormal eigenfunctions $\phi_1^{\rm D},\ldots,\phi_j^{\rm D}$ of $-\Delta_\Omega^{\rm D}$ corresponding to the eigenvalues $\lambda_1^{\rm D},\ldots,\lambda_j^{\rm D}$. According to Lemma~\ref{extrafunction} there is a function $U\in S^1(\Omega)$ that does not lie in the space ${\rm span}\,\{\phi_1^{\rm D},\ldots,\phi_j^{\rm D}\}\cup\mathcal N$ and that satisfies
	$$
	-\Delta U = \lambda_j^{\rm D} U
	\quad\text{in}\ \Omega
	\qquad\text{and}\qquad
	|\nabla U|^2 = \lambda_j^{\rm D} |U|^2
	\quad\text{in}\ \Omega \,.
	$$
	We emphasize that the fact that $U\in S^1(\Omega)$ follows with the help of Lemma \ref{compactness}. Indeed, the facts that $U\in L^\infty(\Omega)$ and that, according to this lemma, $|\Omega|<\infty$ imply that $U\in L^2(\Omega)$. Then the identity for $|\nabla U|$ together with the fact that $U\in S^1_{\rm loc}(G)$ imply that $U\in S^1(\Omega)$.
	
	We shall show that
	\begin{equation}\label{eq:filonovproof}
		\| |\nabla u| \|_{L^2}^2 \leq \lambda_j^{\rm D} \| u \|^2_{L^2}
		\qquad\text{for all}\ u \in {\rm span}\,\{\phi_1^{\rm D},\ldots,\phi_j^{\rm D},U\} \,.
	\end{equation}
	Indeed, we write such a $u$ as
	$$
	u = \sum_{k=1}^j \alpha_k \phi_k^{\rm D} + \alpha_{j+1} U
	$$
	with $\alpha_1,\ldots,\alpha_{j+1}\in\C$ and expand
	$$
	\| |\nabla u| \|_{L^2}^2 = \sum_{1\leq k,k'\leq j} \overline{\alpha_k}\alpha_{k'} \mathcal E[\phi_k^{\rm D},\phi_{k'}^{\rm D}] + 2 \re \sum_{k=1}^j \overline{\alpha_k} \mathcal E[\phi_k^{\rm D},U] + \mathcal E[U,U] \,,
	$$
	where we have set
	$$
	\mathcal E[v,w] := \sum_{\ell=1}^m \int_\Omega \overline{(X_\ell v)} (X_\ell w) \,dx
	$$
	with integration with respect to Haar measure. The eigenvalue equation for $\phi_k^{\rm D}$ and the orthonormality of these functions implies that
	$$
	\mathcal E[\phi_k^{\rm D},\phi_{k'}^{\rm D}] = \lambda_k^{\rm D} \,\delta_{k,k'} \,.
	$$
	Next, the equation for $U$ and the fact that $\phi_k^{\rm D}\in S^1_0(\Omega)$ implies that
	$$
	\mathcal E[\phi_k^{\rm D},U] = \lambda_j^{\rm D} \int_\Omega \overline{\phi_k^{\rm D}} U\,dx \,.
	$$
	Finally, by the choice of $U$,
	$$
	\mathcal E[U,U] = \lambda_j^{\rm D} \int_\Omega |U|^2\,dx \,.
	$$
	Thus, we have
	\begin{align*}
		\| |\nabla u| \|_{L^2}^2 & = \sum_{1\leq k\leq j} \lambda_k^{\rm D} |\alpha_{k}|^2 + 2 \lambda_j^{\rm D} \re \sum_{k=1}^j \overline{\alpha_k} \int_\Omega \overline{\phi_k^{\rm D}} U\,dx +\lambda_j^{\rm D} \int_\Omega |U|^2\,dx \\
		& \leq \lambda_j^{\rm D} \left( \sum_{1\leq k\leq j} |\alpha_{k}|^2 + 2 \re \sum_{k=1}^j \overline{\alpha_k} \int_\Omega \overline{\phi_k^{\rm D}} U\,dx + \int_\Omega |U|^2\,dx \right)
		& = \lambda_j^{\rm D} \| u \|_{L^2}^2 \,,
	\end{align*}
	which proves \eqref{eq:filonovproof}.
	
	Note that since $U$ does not lie in ${\rm span}\,\{\phi_1^{\rm D},\ldots,\phi_j^{\rm D}\}$, we have
	$$
	\dim {\rm span}\,\{\phi_1^{\rm D},\ldots,\phi_j^{\rm D},U\} = j+1 \,.
	$$
	By the variational principle (see, e.g., \cite[Theorem 1.25]{FrLaWe}), \eqref{eq:filonovproof} implies $\lambda_{j+1}^{\rm N} \leq \lambda_j^{\rm D}$. Moreover, if we had equality in this inequality, then it would follow that
	$$
	{\rm span}\,\{\phi_1^{\rm D},\ldots,\phi_j^{\rm D},U\} \subset\mathcal N \,,
	$$	
	which would contradict the fact that $U$ does not lie in $\mathcal N$. Thus we have shown that the inequality is strict and the proof of the theorem is complete.
\end{proof}


\section{An extension}

In this section we shall show that an analogue of Theorem \ref{main} remains valid in a more general setting. As a motivating example we consider the 
Baouendi--Grushin operator 
$$
- \frac{\partial^2}{\partial x^2} - \frac{\partial^2}{\partial y^2} -(x^2+y^2)\frac{\partial^2}{\partial t^2}
$$ 
or the Baouendi--Goulaouic operator
$$
- \frac{\partial^2}{\partial x^2} - \frac{\partial^2}{\partial y^2} -x^2 \frac{\partial^2}{\partial t^2} \,,
$$ 
both acting in $\R^3$ with coordinates denoted by $(x,y,t)$. As before we can consider the Dirichlet and Neumann restrictions to an open set $\Omega\subset\R^3$ and, by considering the trial function $U(x,y,t) = e^{i\sqrt\lambda \omega\cdot(x,y)}$ with $\omega\in\Sph^1$ we can argue as before and find that the $j$-th Dirichlet eigenvalue is strictly larger than the $(j+1)$-st Neumann eigenvalue.

Here is an abstract way to generalize these examples. We follow the presentation of G.~Folland \cite{Fo}, but see also the book of one of the authors with J.~Nourrigat \cite{HeNo1}.

We consider a real vector space $W$ of finite dimension with a direct sum decomposition
$$
W = W_1 \oplus W_2 \oplus \ldots \oplus W_r \,.
$$
Defining dilations $h_t:W\to W$ for $t>0$ by
$$
h_t(\sum_{j=1}^r w_j) = \sum_{j=1}^r t^j w_j
\qquad\text{if}\ w_j\in W_j \,,\ j=1,\ldots,r \,,
$$
we say that a differential operator $P$ on $W$ with smooth coefficients is homogeneous of degree $m$ if
$$
P(f\circ h_t) = t^m (Pf)\circ h_t
\qquad\text{for all}\ f\in C^\infty(W) \,,\ t>0 \,.
$$
Let $X_1,\ldots,X_p$ be smooth, real vector fields on $W$ that are homogeneous of degree 1 and satisfy H\"ormander's condition (meaning the $X_j$'s and their commutators span the tangent space to $W$ at the origin). We are interested in the operator
$$
\mathcal L = \sum_{\ell=1}^p X_\ell^2 \,.
$$

We denote by $\mathfrak g$ the Lie algebra of vector fields generated by $X_1,\ldots,X_p$. This is a stratified algebra with the stratification determined by the homogeneity. We denote
$$
\mathfrak h := \{ X \in\mathfrak g :\ X|_0 = 0 \} \,,
$$
which is a subalgebra of $\mathfrak g$. Our \emph{assumption} is that the codimension of $\{ X\in \mathfrak h :\ X \ \text{is homogeneous of degree 1}\}$ in $\{ X\in \mathfrak g:\ X \ \text{is homogeneous of degree 1} \}$ is at least two.

For an open set $\Omega\subset W$ let $\mathcal L_\Omega^{\rm D}$ and $\mathcal L_\Omega^{\rm N}$ denote the Dirichlet and Neumann realizations of $\mathcal L = \sum_{\ell=1}^p X_\ell^2$ in $L^2(\Omega)$, respectively. The space $L^2(\Omega)$ is defined with respect to Lebesgue measure. Assuming that $\mathcal L_\Omega^{\rm N}$ has compact resolvent, we introduce the eigenvalues $\lambda_j(\mathcal L_\Omega^{\rm D})$ and $\lambda_j(\mathcal L_\Omega^{\rm N})$, $j\in\N$, similarly as before.

\begin{theorem}
	\label{main2}
	Under the above assumptions, $\lambda_j(\mathcal L_\Omega^{\rm D})>\lambda_{j+1}(\mathcal L_\Omega^{\rm N})$ for all $j\in\N$.
\end{theorem}

Before giving the proof of this theorem, let us show that the two examples discussed at the beginning of this section fall in our general framework. Indeed, in both examples we consider the variables $x$ and $y$ to be of homogeneity one and the variable $t$ of homogeneity two. The Lie algebra $\mathfrak g$ is spanned by $\partial/\partial x$, $\partial/\partial y$, $x\partial/\partial t$ and $y\partial/\partial t$ in the first example and by the first three in the second example. The subspace of vector fields in $\mathfrak h$ that are homogeneous of degree 1 is spanned by $x\partial/\partial t$ and $y\partial/\partial t$ in the first example and by $x\partial/\partial t$ in the second example. Thus, in both cases this space is of codimension 2 in the subspace of vector field in $\mathfrak g$ that are homogeneous of degree~1.

\begin{proof}
	\emph{Step 1.} Let $G$ be the Carnot group corresponding to $\mathfrak g$ and let $\widetilde X_1,\ldots,\widetilde X_p$ be the left-invariant vector fields on $G$ corresponding to the vector fields $X_1,\ldots,X_p$ on $W$. We consider the sub-Laplacian $-\Delta = - \sum_{\ell=1}^p \widetilde X_\ell^2$ on $G$.
	
	It follows from the H\"ormander condition that the codimension of $\mathfrak h$ in $\mathfrak g$ is equal to $\dim W=:k$. Since $\mathfrak h$ is a subalgebra of $\mathfrak g$, there is an operator
	$$
	\pi_{(0,\mathfrak h)}(-\Delta)
	$$
	on $\R^k$. For the definition we refer to \cite[Subsection 1.3]{HeNo1}.
	
	Folland has shown that there is a diffeomorphism $\theta:W\to\R^k$ with Jacobian equal to 1 such that $f\in\mathcal S(\R^k)$ if and only if $f\circ\theta\in\mathcal S(W)$, and in this case
	$$
	\pi_{(0,\mathfrak h)}(-\Delta)[f] = ( \mathcal L[f\circ\theta])\circ\theta^{-1} \,.
	$$
	(This is essentially \cite[Proposition 1.4.1]{HeNo1}, except that they choose $\mathfrak g$ as a free nilpotent group, while we choose it here in a minimal fashion. This difference is discussed in Folland's paper \cite{Fo}.)
	
	The upshot of this discussion is that instead of the operator $\mathcal L$ in $W$ we can consider the operator $\pi_{(0,\mathfrak h)}(-\Delta)$ in $\R^k$.
	
	\bigskip
	
	\emph{Step 2.} We denote by $V_1$ and $\mathfrak h_1$ the subspaces of $\mathfrak g$ and $\mathfrak h$ of elements that are homogeneous of degree 1.  We can now choose coordinates in $\mathfrak g$ as in  \cite[Equation (1.2.6)]{HeNo1} and a new basis for $V_1$ corresponding to a direct sum decomposition $V_1= E_1 \oplus  \mathfrak h_1$. We then express $\pi_{(0,\mathfrak h)}(\widetilde X_\ell)$ as in \cite[Equation (1.3.2)]{HeNo1} with $\sigma'(t,a)$ given in \cite[Equation (1.3.4)]{HeNo1}. We are interested in the computation of $\sigma'(t,a)$ for $a \in V_1$.
	
	We assume that we are given the Euclidean structure on $V_1$. Let $k_1$ denote the codimension of $\mathfrak h_1$ in $V_1$.
	We assume that $\widetilde X_\ell$ ($\ell=1,\ldots, k_1$) is a basis of $E_1$ and $\widetilde X_\ell$ ($\ell=k_1+1,\ldots, p$) is a basis of $\mathfrak h_1$. Following the construction in \cite{HeNo1}, we have
	$$
	\pi_{(0,\mathfrak h)} (\widetilde X_\ell) = \frac{\partial}{\partial t_\ell} + \sum_{j \geq k_1+1} P_{\ell,j} (t) \frac{\partial}{\partial t_j} \,,\qquad\forall \ell=1,\ldots, k_1\,,
	$$
	and 
	$$
	\pi_{(0,\mathfrak h)} (\widetilde X_\ell)= \sum_{j\geq k_1+1} P_{\ell,j} (t) \frac{\partial}{\partial t_j} \,,\qquad \forall \ell=k_1+1,\ldots,p\,,
	$$
	where the $P_{\ell,j}$ have the appropriate homogeneity (which in particular implies that they vanish at $0$).
	
	Our assumption is that $k_1\geq 2$. In particular, when restricted to functions $f$ depending only on $(t_1,t_2)$ we have 
	$$
	\pi_{(0,\mathfrak h)} (-\Delta) f=\left(-\frac{\partial^2}{\partial t_1^2} - \frac{\partial^2}{\partial t_2^2}\right)f\,.
	$$
	Once we have this, we can argue as before.
\end{proof}

The framework described in this section encompasses many more examples than the two that we have already mentioned. Obvious generalizations are
$$
-\Delta_x - |x|^{2a} \Delta_y
$$
in $\R^n\times\R^m$ with coordinates $(x,y)$ with $a\in\N$ and $n\geq 2$. The latter condition guarantees our codimension assumption. Meanwhile it is an open problem whether Theorem \ref{main2} remains valid without the codimension assumption. This is even unclear in the special cases of the above operator with $n=1$ or for the operator
$$
-\frac{\partial^2}{\partial x_1^2} - \left(x_1 \frac\partial{\partial x_2} + x_2 \frac\partial{\partial x_3} + \ldots + x_{n-1} \frac\partial{\partial x_n} \right)^2 \,.
$$


\bibliographystyle{amsalpha}

\end{document}